\documentclass[11pt]{amsart}
\usepackage{amssymb,amsmath,amsthm,bbm,enumerate,mathtools}
\usepackage{color}

\newtheorem{theorem}{Theorem}[section]
\newtheorem{lemma}[theorem]{Lemma}
\newtheorem{corollary}[theorem]{Corollary}
\newtheorem{proposition}[theorem]{Proposition}
\theoremstyle{remark}
\newtheorem{remark}[theorem]{Remark}

\renewcommand{\epsilon}{\varepsilon}
\renewcommand{\phi}{\varphi}
\renewcommand{\kappa}{\varkappa}
\renewcommand{\tilde}{\widetilde}

\DeclarePairedDelimiter\ceil{\lceil}{\rceil}
\DeclarePairedDelimiter\floor{\lfloor}{\rfloor}

\newcommand{\IE}{\mathbb{E}}
\newcommand{\I}{\mathbbm{1}}
\newcommand{\IN}{\mathbb{N}}
\newcommand{\IP}{\mathbb{P}}
\newcommand{\IR}{\mathbb{R}}
\newcommand{\IZ}{\mathbb{Z}}

\newcommand{\tB}{B^0}
\newcommand{\tC}{C^0}
\newcommand{\tX}{X^0}
\newcommand{\cal}{\mathcal}
\newcommand{\meas}{{\text{meas}}}
\newcommand{\eps}{{\varepsilon}}

\begin{document}
\title[Excited random walks]{Excursions and occupation times of\\
  critical excited random walks}

\author{Dmitry Dolgopyat and Elena Kosygina}
\begin{abstract}
  We consider excited random walks (ERWs) on integers in i.i.d.\
  environments with a bounded number of excitations per site.  The
  emphasis is primarily on the critical case for the transition
  between recurrence and transience which occurs when the total
  expected drift $\delta$ at each site of the environment is equal to
  1 in absolute value. Several crucial estimates for ERWs fail in the
  critical case and require a separate treatment.  The main results
  discuss the depth and duration of excursions from the origin for
  $|\delta|=1$ as well as occupation times of negative and positive
  semi-axes and scaling limits of ERW indexed by these occupation
  times. We also point out that the limiting proportions of the time spent
  by a non-critical recurrent ERW (i.e.\ when $|\delta|<1$) above or
  below zero converge to beta random variables with explicit
  parameters given in terms of $\delta$. The last observation can be
  interpreted as an ERW analog of the arcsine law for the
  simple symmetric random walk.
\end{abstract}
\maketitle

\section{Introduction and main results}

\subsection{Model description}

We consider an exited random walk (ERW) on $\IZ$ with nearest neighbor
jumps which evolves in a random ``cookie environment''. Each site of
the lattice contains a stack of ``cookies''
$\omega_x:=(\omega_x(1),\omega_x(2),.\dots)$. A cookie
$\omega_x(i)\in [0,1]$, $x\in\IZ$, $i\in\IN$, encodes the probability
that the walk jumps to the right upon the $i$-th visit to $x$. We
assume that the cookie stacks $\omega_x$, $x\in\IZ$, are spatially
i.i.d.\ and that there is a non-random $M\ge 0$, the number of
excitations per site, such that $\omega_x(i)=1/2$ for all $i>M$ and
$x\in\IZ$, i.e.\ starting from the $(M+1)$-th visit to a site the walk
makes only unbiased jumps from this site.

More formally, we suppose that an environment
$\omega\in\Omega=[0,1]^{\IZ\times\IN}$ is chosen according to a
probability measure $\IP$ which satisfies the following three
assumptions.
\begin{itemize}
\item [(IID)] (Independence) The cookie stacks $\omega_x(\cdot)$,
  $x\in\IZ$, are i.i.d.\ under $\IP$.
\item [(WEL)] (Weak ellipticity) For all
  $x\in\IZ$ \[\IP(\omega_x(i)>0\ \forall
  i\in\IN)>0\quad\text{and}\quad \IP(\omega_x(i)<1\ \forall i\in\IN)>0.\]
\item [($\mathrm{BD_M}$)] (Bounded number of excitations per site)
  $\IP(\omega_x(i)=1/2\ \forall x\in\IZ,\ i>M)=1$.
\end{itemize}
Given an environment $\omega\in\Omega$, we shall use the usual
coin-toss construction of a random walk, albeit we should keep a
record of the number of visits of the walk to each site and use
appropriately biased coins for the first $M$ visits to each
site. Namely, let $(\eta_x(i))_{x\in\IZ,i\in\IN}$ be independent
(under some probability measure $P_\omega$) Bernoulli random variables
such that $P_\omega(\eta_x(i)=1)=1-P_\omega(\eta_x(i)=0)=\omega_x(i)$
for all $x\in\IZ$, $i\in\IN$.  Set $X_0=x$, $x\in\IZ$, and define
recursively \[X_{n+1}=X_n+2\eta_{X_n}(\#\{k\in\{0,1,\dots,n\}:\,X_k=X_n\})-1,
\quad n\in\{0\}\cup\IN.\] The probability measure $P_{\omega,x}$
induced on the space of random walk paths which start from $x$ is
called the {\em quenched} measure. The probability measure on the
product space of environments and random walk paths originating at $x$
defined by \[P_x(\cdot)=\IE
\left[P_{\omega,x}(\cdot)\right]=\int_\Omega
P_{\omega,x}(\cdot)\,d\IP(\omega)\] is called the {\em averaged}
measure. Observe that ERW is not a Markov process with respect to
either of these measures.

Below we shall only quote the facts needed to put our results into
the context of previous work. For an overview of various ERW models,
methods, and results the reader is referred to \cite{KZ13}. 

\subsection{Excursions from the origin} Let $T_k:=\inf\{n\ge 0:
X_n=k\}$, $k\in\mathbb{Z}$, be the time of the first visit to $k$ and
$T^r_0:=\inf\{n\ge 1:\ X_n=0\}$ be the first strictly positive time at
which the random walk visits the origin.

Under our assumptions, several phase transitions are known to be
characterized by
\begin{equation}
  \label{del}
  \delta:=\IE\left[\sum_{i=1}^M(2\omega_0(i)-1)\right],
\end{equation}
the expected total drift stored in a single cookie stack.  
The excited random walk $(X_n)_{n\ge 0}$
\begin{enumerate}[(i)]
\item is transient, i.e.\ $|X_n|\to \infty$ $P_0$-a.s., iff $|\delta|>
  1$ (see \cite[Theorem 3.10]{KZ13} and the references therein or a
  combination of \cite[Corollary 7.10]{KZ14} and Remark~\ref{rt}
  below)\footnote{for $|\delta|\le 1$ $X$ is recurrent, i.e.\ returns
    to the origin infinitely often $P_0$-a.s..};
\item is ballistic, i.e.\ there is a constant $v\ne 0$ such that
  $P_0$-a.s.\ $\lim\limits_{n\to\infty}X_n/n=v$, iff $|\delta|>2$ (see
  \cite[Theorem 5.2]{KZ13} and the references therein);
\item is strongly transient, i.e.\
  $E^0\left[T^r_0\,|\,T^r_0<\infty\right]<\infty$, iff $|\delta|>3$ (see
  \cite[Corollary 1.2]{KZ14});
\item after diffusive scaling converges under $P_0$ to a Brownian
  motion iff $|\delta|>4$ or $\delta=0$ (see \cite[Theorems 6.1, 6.3, 6.5,
  6.7]{KZ13} and the references therein).
\end{enumerate}

\begin{remark}
  The velocity $v$ in (ii) as well as all constants $b,c,c_i$, $i\ge
  1$, which appear below depend on the distribution of a single cookie
  stack $\omega_0$ under $\IP$. They are not, in general, functions of
  $\delta$ (see \cite[Remark 5.8]{KZ13} for a discussion about $v$).
\end{remark}
The phase transition in (iii) emerged in the study of the depth and
duration of excursions of ERW. Since our first result is about
excursions in the critical case $|\delta|=1$ we shall first quote the
original relevant theorem.
\begin{theorem}[\cite{KZ14}, Theorem 1.1] \label{kz14}
  Assume that $\delta\in\IR\setminus\{1\}$. Then there are constants
  $c_1,c_2\in(0,\infty)$ such that
  \begin{align}\label{depth}
&\lim_{n\to\infty}n^{|\delta-1|}P_1(T_n<T_0<\infty)=c_1,\\\label{timed}
&\lim_{n\to\infty}n^{|\delta-1|/2}P_1(n<T_0<\infty)=c_2.
  \end{align}
Moreover, if $\delta=1$ then every $\epsilon>0$,
\begin{equation}
  \label{del1}
  \lim_{n\to\infty}n^\epsilon P_1(T_n<T_0)=\lim_{n\to\infty}n^\epsilon P_1(T_0>n)=\infty.
\end{equation}
If $|\delta|\ne 1$\footnote{In (1.5) of \cite{KZ14} both $\delta=1$ and $\delta=-1$ should have been excluded.} then there is a constant $c_3\in(0,\infty)$ 
such that
\begin{equation}
  \label{ret}
  \lim_{n\to\infty}n^{||\delta|-1|/2}P_0(n<T_0^r<\infty)=c_3.
\end{equation}
Moreover, if $|\delta|=1$ then for every $\epsilon>0$,
\begin{equation}
  \label{ret0}
  \lim_{n\to\infty}n^\epsilon P_0(T_0^r>n)=\infty.
\end{equation}
\end{theorem}
This theorem immediately implies (iii) but provides very
little information about the tail of the return time in the critical case. Our first result fills in this gap.
\begin{theorem}\label{main_res1}
If $\delta=1$ then
there is a constant $c_4\in (0,\infty)$ such that 
\begin{align}
  \label{depth1}&\lim_{n\to\infty}(\ln n)\,P_1(T_n<T_0)=c_4;
\\\label{time1}
&\lim_{n\to\infty}(\ln n)\,P_1(T_0>n)=2c_4.
\end{align}
Moreover, if $|\delta|=1$ then  
\begin{equation}\label{ret1}
\lim_{n\to\infty}(\ln n)\, P_0(T^r_0>n)=c_5:=
\begin{cases}
  2c_4\IE[\omega_0(1)],& \text{if $\delta=1$};\\2c_4\IE[1-\omega_0(1)], & \text{if $\delta=-1$.} 
\end{cases}
\end{equation}
\end{theorem}
The key statements of Theorem~\ref{main_res1} are (\ref{depth1}) and
(\ref{time1}). The last conclusion follows easily from
(\ref{time1}), (\ref{timed}) with $\delta=-1$, and the following
remark by conditioning on the first step (see \cite[(6.2)]{KZ14}).
\begin{remark}\label{sym}
  There is a useful symmetry in our model. If the environment
  $(\omega_x)_{x \in \IZ}$ is replaced with $(\tilde{\omega}_x)_{x \in
    \IZ} $ where $\tilde{\omega}_x (i)=1- \omega_x (i)$, for all
  $i\in\IN,\ x\in\IZ$, then $\tilde{X}$, the ERW corresponding to the
  new environment, satisfies
\begin{equation}
\tilde{X}\overset{\mathrm{d}}{=} -X,
\end{equation}
where $\overset{\mathrm{d}}{=}$ denotes the equality in distribution.
Thus, it is sufficient to consider only excursions to the right (for
all $\delta$). The corresponding results for excursions to the left
will follow by symmetry.  Thus from now on
we shall assume without loss of generality that $\delta\ge 0$.
\end{remark}

\subsection{Occupation times and scaling limits} Unless stated otherwise we shall assume
that all processes start at the origin at time $0$. Let $B=(B(t)),\
t\ge 0$, denote a standard Brownian motion and
$W_{\alpha,\beta}=(W_{\alpha,\beta}(t)),\ t\ge 0$, be an $(\alpha,
\beta)$-perturbed Brownian motion, i.e.\ the solution of the equation
\begin{equation}
  \label{abp}
  W_{\alpha,\beta}(t)=B(t)+\alpha\sup_{s\leq t}
W_{\alpha,\beta}(s)+\beta \inf_{s\leq t} W_{\alpha,\beta}(s).
\end{equation}
Reflected $\alpha$-perturbed Brownian motion,
$W_\alpha=(W_\alpha(t)),\ t\ge 0$, is the solution of
\begin{equation}
  \label{rap}
  W_\alpha(t)=B(t)+\alpha\sup_{s\leq t} W_\alpha(s)+\frac{1}{2} L^{W_\alpha}(t), 
\end{equation}
where $L^{W_\alpha}(t)$ is the local time of $W_\alpha$ at
zero. Equation (\ref{abp}) has a path-wise unique solution if
$(\alpha,\beta)\in(-\infty,1)\times(-\infty,1)$, and (\ref{rap}) has a
path-wise unique solution when $\alpha<1$ (\cite{Da96, PW, CD99}). In both
cases the solution is adapted to the filtration of $B$. If $\beta=0$
then the solution of (\ref{abp}) can be written explicitly:
\begin{equation}
  \label{bzero}
  W_{\alpha,0}(t)=B(t)+\frac{\alpha}{1-\alpha}\,\sup_{s\le t}B(s).
\end{equation}

Throughout the paper we use $\ \Rightarrow\ $ to denote the weak
convergence of random variables and $\ \overset{J_1}{\Rightarrow}\ $ for
the weak convergence of stochastic processes with respect to the
standard Skorokhod topology $J_1$ on $D([0,\infty))$, the space of
c\`adl\`ag functions on $[0,\infty)$.\footnote{Since all limiting
  processes below have continuous paths, we can also claim the
  convergence with respect to the uniform topology on $D([0,T])$ for
  each $T>0$ (see \cite[Section 15]{B99}).} 

The following two theorems describe scaling limits of recurrent ERWs.
\begin{theorem}[\cite{DK12}, Theorem 1.1]
\label{ThERWRecLim}
Let $\delta\in[0,1)$. Then under $P_0$ \[\dfrac{X_{[n\cdot]}}{\sqrt
  n}\overset{J_1}{\Rightarrow} W_{\delta,-\delta}(\cdot)\ \text{as }n\to\infty. \]
\end{theorem}

\begin{theorem}[\cite{DK12}, Theorem 1.2]
\label{ThERWCrit}
Let $\delta=1$ and $B^*(t):=\max_{s\le t}B(s)$. Then there exists a constant $b\in(0,\infty)$ such that under $P_0$
\[\frac{X_{[n\cdot]}}{b \sqrt{n}\log n}\overset{J_1}{\Rightarrow} B^*(\cdot)\
\text{as }n\to\infty.\]
\end{theorem}
At the first glance it appears counter-intuitive that for $\delta=1$
the limiting process is transient while the original process is
recurrent. However, the running maximum of Brownian motion is a
natural limit of $W_{\alpha, \beta}((1-\alpha)^2\,\cdot)$ as
$\alpha\uparrow 1$ (see the discussion right after Theorem
\ref{PrPBMPOsNeg}).
 
Theorem~\ref{ThERWRecLim} suggests that the rescaled occupation times
of positive and negative semi-axes of non-critical recurrent ERW
should converge to those of the reflected perturbed Brownian motion. The
latter was studied in detail, and the next theorem quotes results from
the literature. Let \[A^+(t):=\int_0^t \I_{\{W_{\alpha,\beta}(u)\ge
  0\}}\, du,\ \ A^-(t):=\int_0^t \I_{\{W_{\alpha,\beta}(u)<0\}}\, du,
\ t\ge 0,\] and $T_{\pm}(t):=\inf\{s:\, A^\pm(s)>t\}$, $t\ge 0$, be
the right continuous inverses of $A^\pm(\cdot)$. Denote by $Z(a,b)$ a
$\beta$-distributed random variable with parameters $a$ and $b$.
\begin{theorem}
\label{PrPBMPOsNeg} For all $\alpha,\beta<1$ the following holds:
\begin{itemize}
\item [(a)] \cite[equation (8)]{CPY}\[\dfrac{A^+(t)}{t}\overset{\mathrm{d}}{=}Z\left(\dfrac{1-\beta}{2},
    \dfrac{1-\alpha}{2}\right)\ \ \text{and}  
  \ \ \dfrac{A^-(t)}{t}\overset{\mathrm{d}}{=}
  Z\left(\dfrac{1-\alpha}{2},\dfrac{1-\beta}{2}\right).\]
\item [(b)] \cite[Theorem 1]{CD00}\[ \quad
  W_{\alpha,\beta}(T^+(\cdot))\overset{\mathrm{d}}{=}W_\alpha(\cdot)\ \ 
  \text{and}
  \ \ -W_{\alpha,\beta}(T^-(\cdot))\overset{\mathrm{d}}{=}W_{\beta}(\cdot).\]
\end{itemize}
\end{theorem} 
Theorem~\ref{PrPBMPOsNeg} implies that $W_{\alpha, \beta}((1-\alpha)^2
\,\cdot)\Rightarrow B^*(\cdot) \text{ as }\alpha\uparrow 1$.  Indeed,
the Brownian scaling of $W_{\alpha,\beta}$
(\cite[Proposition~2.3]{CPY}) allows to rewrite the above convergence as
 \begin{equation}
 \label{LimABPert2}
 (1-\alpha) W_{\alpha, \beta}(\cdot)\Rightarrow  
 B^*(\cdot) \text{ as }\alpha\uparrow 1. 
 \end{equation}
 By Theorem~\ref{PrPBMPOsNeg}(a) $\alpha\uparrow 1$ the process
 $W_{\alpha, \beta}$ stays most of the time in $[0,\infty)$ 
 (recall that
 $\IE\left[A^+(t)/t\right]=(1-\beta)/(2-\alpha-\beta)$). By
 Theorem~\ref{PrPBMPOsNeg}(b) we conclude that the limit in
 \eqref{LimABPert2} should be independent of $\beta$. On the other
 hand, if $\beta=0$ then (\ref{bzero}) tells us that $(1-\alpha)
 W_{\alpha, 0}(\cdot)$ has the same law as $(1-\alpha) B(\cdot)+\alpha
 \sup_{s\leq \cdot} B(s)$. This implies \eqref{LimABPert2}.

The next corollary follows from Theorems~\ref{ThERWRecLim} and
\ref{PrPBMPOsNeg} by the continuous mapping theorem (see
Section~\ref{four} for details).
\begin{corollary}
\label{CrPBMPOsNeg}
Suppose that $\delta\in[0,1)$. Let 
\[A^+_n:=\sum_{i=0}^n \mathbbm{1}_{\{X_i\ge 0\}}\ \ \text{and}\ \
A^-_n:=\sum_{i=0}^n \mathbbm{1}_{\{X_i<0\}},\ n\ge 0, \] and
$T^\pm_m:=\inf\{n\ge 0:\, A^\pm_n>m\}$, $m\ge 0$. Then
\begin{alignat*}{4}
  &(a)\quad\frac{A^+_n}{n}\Rightarrow
  Z\left(\frac{1+\delta}{2},\frac{1-\delta}{2}\right)\quad
  &&\text{and}\quad &&\frac{A^-_n}{n}\Rightarrow
  Z\left(\frac{1-\delta}{2},\frac{1+\delta}{2}\right)\ &&\text{as $n\to\infty$};\\
  &(b)\quad\quad\quad
  \frac{X_{T^+_{\floor{m\cdot}}}}{\sqrt{m}}\overset{J_1}{\Rightarrow}
  W_\delta(\cdot)\ \ &&\text{and}\ \
  &&\quad -\frac{X_{T^-_{\floor{m\cdot}}}}{\sqrt{m}}\overset{J_1}{\Rightarrow}
  W_{-\delta}(\cdot)\ &&\text{as $m\to\infty$}.
\end{alignat*}
\end{corollary} 

Consider now the critical case $\delta=1$. It is clear from
Theorem~\ref{ThERWCrit} that the proportion of time spent in
$(-\infty,0)$ by an ERW with $\delta=1$ should converge to $0$ (see
Lemma~\ref{time_pos} below). Since the critical ERW is recurrent and
satisfies ($\mathrm{BD_M}$), $A^-_n\to\infty$ as $n\to\infty$. But
how fast does $A^-_n$ increase? Our last theorem answers this
question on a logarithmic scale and also provides scaling limits of
$X_{T^\pm_{\floor{m\cdot}}}$ when $\delta=1$.

\begin{theorem}\label{main_res2}
  Let $\delta=1$ and $A^\pm_n,\ T^\pm_n$, $n\ge 0$, be as in
  Corollary~\ref{CrPBMPOsNeg}. Then under $P_0$
\begin{itemize}
\item [(a)] $\dfrac{\log{A^-_n}}{\log n}\Rightarrow U$ as
  $n\to\infty$, where $U$ is uniform on $[0,1]$ random variable;
\item [(b)] $-\dfrac{X_{T^-_{\floor{m\cdot}}}}{\sqrt{m}}\overset{J_1}{\Rightarrow}
  W_{-1}(\cdot)$ as $m\to\infty$;
\item [(c)] there is a constant $b\in(0,\infty)$ such that
  $\dfrac{X_{T^+_{\floor{m\cdot}}}}{b \sqrt{m}\log m}\overset{J_1}{\Rightarrow} B^*(\cdot)$ as $m\to\infty$.
\end{itemize}
\end{theorem}

Part (a) of the above theorem informally says that $A^-_n\asymp n^U$
where $U$ is a standard uniform random variable.  See
Section~\ref{4.2} for a heuristic derivation of this asymptotics. Part
(b) is just a simple extension of the last claim of
Corollary~\ref{CrPBMPOsNeg} to $\delta=1$. This reflects the fact that
if we consider an ERW with $\delta=1$ only at the times when it visits
the negative half-line then such process is not critical and can be
treated essentially in the same way as the case $\delta\in[0,1)$.  The
situation is different if we look at an ERW with $\delta=1$ only when
it visits the positive half-line, since neither $W_{\alpha,\beta}$ nor
$W_\alpha$ exists for $\alpha=1$. But in view of
Theorem~\ref{ThERWCrit} the statement of part (c) is not surprising.

\subsection{Organization of the paper} In Section~\ref{cbp} we explain
the connection between ERWs and some branching processes. The main
theorem of Section~\ref{cbp}, Theorem~\ref{tbp}, is an important tool
for the proofs of our main results. We illustrate this by deriving
Theorem~\ref{main_res1} as a simple corollary of
Theorem~\ref{tbp}. The proof of Theorem~\ref{tbp} is given in
Section~\ref{ptbp}. In Section~\ref{four} we prove
Corollary~\ref{CrPBMPOsNeg} and Theorem~\ref{main_res2}. Proofs of
technical lemmas are collected in the Appendix.
\section{Connection with branching processes}\label{cbp}

In this section we construct the relevant branching process (BP) and
restate (\ref{depth1}) and (\ref{time1}) in terms of of the
tails of the extinction time and the total progeny of these BPs.

We shall use the same environment $\omega\in \Omega$ and Bernoulli
random variables $(\eta_x(i))_{x\in\IZ,i\in\IN}$ as in the
construction of the ERW. This will provide us with a natural coupling
between the ERW and the BP. We define here only the BP
$V$ which corresponds to right excursions of the walk.\footnote{The BP
corresponding to left excursions, $V^-$, is constructed in a symmetric
way and will be introduced in Section~\ref{4.2}.}  For $x\in \{0\}\cup
\IN$ let
\[S_x(0)=0,\ \ S_x(m):=\inf\left\{k\ge 1:\
  \sum_{i=1}^k(1-\eta_x(i))=m\right\}-m,\ \ m\in\IN.\] Thus, $S_x(m)$ is
the number of ``successes'' before the $m$-th ``failure'' in the sequence
$\eta_x(i)$, $i\in\IN$.  Define the process $V=(V_n)_{n\ge 0}$ which
starts with $y$ particles in generation $0$ by
\begin{equation}
  \label{V}
  V_0=y,\ \
V_n=S_n(V_{n-1}),\ \ n\in\IN.
\end{equation}
If there were no biased coins, $V$ would be a Galton-Watson process
with mean 1 geometric offspring distribution. Our process uses up to
$M$ possibly biased coins in each generation, therefore, strictly
speaking, it is not a ``true'' branching process. We could recast it
as a branching process with migration (see \cite[Section 3]{KZ08})
but, since we do not use any results from branching processes
literature, we shall not need this step.

For $y\in[0,\infty)$ we shall denote by $P^V_y$ the (averaged)
probability measure corresponding to the process $V$ which starts with
$\floor{y}$ particles in generation $0$.  For $x\in [0,\infty)$ define
$\tau_x^V:=\inf\{n\in\mathbb{N}:\,V_n\ge x\}$ and
$\sigma_x^V:=\inf\{n\in\mathbb{N}:\,V_n\le x\}$. When there is no
danger of confusion we shall drop the superscript $V$.

\begin{theorem}\label{tbp}
  Let $\delta=1$ and $V=(V_n)_{n\ge 0}$ be defined by (\ref{V}).  Then
  for each $y\in\mathbb{N}$ there is a constant $c_6(y)\in(0,\infty)$
  such that
\begin{align}\label{aet}
\lim_{n\to\infty}(\ln n)\, P^V_y(\sigma_0^V>n)&=c_6(y),\\
\lim_{n\to\infty}(\ln n)\, P^V_y\left(\sum_{i=0}^{\sigma_0^V-1}V_i>n\right)&=2c_6(y).
\label{atp}
\end{align}
\end{theorem}

Assume for the moment Theorem~\ref{tbp} and derive
Theorem~\ref{main_res1}.
\begin{proof}[Proof of Theorem~\ref{main_res1}]
  The proof is essentially the same as that of Theorem~1.1 in
  \cite{KZ14}. Let the ERW start with $x=1$ and the corresponding BP
  start with $y=1$. Observe that, since ERW and BP are constructed
  from the same $(\eta_x(i))_{x\in\IZ,i\in\IN}$, we
  have \[\sigma_0^V=\max\{X_n\,:\,n<T_0\}\
  \ \text{and}\ \
  T_0\mathbbm{1}_{\{T_0<\infty\}}=\bigg(2\sum_{n=0}^{\sigma_0^V-1}V_n
  - 1\bigg)\mathbbm{1}_{\{\sigma_0^V<\infty\}}.\] Therefore,
  (\ref{depth1}) and (\ref{time1}) with $c_4=c_6(1)$ follow from
  (\ref{aet}) and (\ref{atp}).  To show (\ref{ret1}) we start ERW with
  $x=0$ and condition on the first step. Since $P_{\omega,\pm 1}(T_0\ge n)$
  do not depend on $\omega_0(\cdot)$,
  \begin{align*}
    P_0(T^r_0>n)&=\IE[P_{\omega,0}(T^r_0>n)]\\
    &=\IE[\omega_0(1)P_{\omega,1}(T_0\ge
    n)]+\IE[(1-\omega_0(1))P_{\omega,-1}(T_0\ge n)]\\
    &=\IE[\omega_0(1)]P_1(T_0\ge n)+\IE[(1-\omega_0(1))]P_{-1}(T_0\ge
    n).
  \end{align*}
  By (WEL), $\IE[\omega_0(1)]>0$ and $\IE[(1-\omega_0(1))]>0$. If
  $\delta=1$ then $(\ln n)P_1(T_0\ge n)\to 2c_4$ as $n\to\infty$ by
  (\ref{time1}). By Remark~\ref{sym} and (\ref{timed}) with
  $\delta=-1$, $nP_{-1}(T_0\ge n)$ converges to a constant. We
  conclude that 
  \begin{equation}
  \label{C4C5}
  \lim_{n\to\infty} (\ln
  n)P_0(T^r_0>n)=2c_4\IE[\omega_0(1)].
  \end{equation}
   The result for $\delta=-1$
  follows by symmetry.
\end{proof}

\section{Proof of Theorem~\ref{tbp}}\label{ptbp}
The proof of Theorem~\ref{tbp} depends on a number of additional facts
which we state below and prove in the Appendix.

\begin{lemma}\label{hitn}
  Let $\delta=1$ and 
  $y\in\mathbb{N}$. Then there is a constant $c_6(y)\in(0,\infty)$
  such that \[\lim_{n\to\infty}(\ln n)\,P_y(\tau_n<\sigma_0)=c_6(y).\]
\end{lemma}

\begin{lemma}\label{delay}
  Let $\delta=1$. For every $y\in\mathbb{N}$ and $\alpha>1$ \[\lim_{n\to\infty}
  (\ln n)P_y\left(\sum_{i=0}^{\sigma_0-1}\mathbbm{1}_{\{V_i\le 
      n\}}>n^\alpha\right)=0. \]
\end{lemma}

\begin{lemma}\label{one}
  Let $\delta=1$. For every
  $h>0$
  \begin{align}
\label{time0} &\lim_{n\to\infty}P_n(\sigma_0>hn)=1;\\
    \label{survive}&\lim_{n\to\infty}P_n\left(\sum_{i=0}^{\sigma_0-1}V_i>hn^2\right)=1.
  \end{align}
\end{lemma}

The following results, which will be referred to as (DA), Diffusion
Approximation, and (OS), ``Overshoot'', respectively, are borrowed from
previous works.
\begin{lemma}[Diffusion approximation]\label{DA}
  Let $\delta=1$. Fix an arbitrary $\epsilon>0$ and $y>\epsilon$. Let
  $Y^{\epsilon,n}(0)=[ny]$ and $Y^{\epsilon,n}(t)=\dfrac{V_{[nt]\wedge
      \sigma_{\varepsilon n}}}{n}$, $t\ge 0$. Then, under the averaged
  measure, $Y^{\epsilon,n}\overset{J_1}{\Rightarrow} Y$, where $Y$ is
  the solution of
\begin{equation}
\label{SqB2}
dY(t)=dt+\sqrt{2Y(t)}\, dB(t), \quad Y(0)=y,
\end{equation}
stopped when $Y$ reaches level $\epsilon$.
\end{lemma}
Lemma~\ref{DA} is an immediate consequence of Proposition 3.2 and
Lemma 3.3 of \cite{KZ14}. 

\begin{lemma}[``Overshoot'', Lemma 5.1 of \cite{KM}]\label{OS}
  There are constants $c_7, c_8>0$ and $N\in\mathbb{N}$ such that for all
  $x\geq N$ and $y\geq 0$
\[\max_{0\leq z<x} P_z(V_{\tau_x}>x+y\,|\,\tau_x<\sigma_0)\leq c_7 \left(e^{-c_8 y^2/x}+e^{-c_8 y}\right) \]
and
\[ \max_{x<z<4x} P_z(V_{\sigma_x\wedge \tau_{4x}}<x-y)\leq c_7 e^{-c_8 y^2/x}.\]
\end{lemma}

\begin{proof}
  [Proof of Theorem~\ref{tbp}] We start with the proof of (\ref{aet}).

  {\em Lower bound for (\ref{aet}).} For every $y\in \mathbb{N}$ we
  have by the strong Markov property and monotonicity in the starting
  point that
  \begin{multline*}
    P_y(\sigma_0>n)\ge
    P_y(\sigma_0>n,\tau_n<\sigma_0)\\=P_y(\sigma_0>n\,|\,\tau_n <\sigma_0)
    P_y(\tau_n<\sigma_0)\ge P_n(\sigma_0>n)P_y(\tau_n<\sigma_0).
  \end{multline*}
  Using Lemma~\ref{hitn} and
  (\ref{time0}) we get
  \begin{equation*}
\liminf_{n\to\infty}(\ln n)P_y(\sigma_0>n)\ge c_6(y).    
  \end{equation*}

  {\em Upper bound for (\ref{aet}).} Fix an arbitrary $\alpha>1$ and
  notice that for all $m>y$
\begin{align*}
  (\ln m)\, P_y(\sigma_0>m^\alpha)&\le(\ln m)\,
  P_y(\sigma_0>m^\alpha,\tau_m\le
  m^\alpha)+(\ln m)\,P_y(\tau_m\wedge \sigma_0 >m^\alpha)\\
  & \le (\ln m)\,P_y(\sigma_0>\tau_m)+(\ln
  m)\,P_y\left(\sum_{i=1}^{\sigma_0-1}\mathbbm{1}_{\{V_i\le
      m\}}>m^\alpha\right).
\end{align*}
As $m\to\infty$, the first term in the right hand side converges to
$c_6(y)$ by Lemma~\ref{hitn} and the second term vanishes due to
Lemma~\ref{delay}. 

Define $m=m(n)$ by the condition $m^\alpha\le n<(m+1)^\alpha$. Then we
get \[\limsup_{n\to\infty}(\ln n)\, P_y(\sigma_0>n)\le
\lim_{\alpha\downarrow 1}\lim_{m\to \infty}\alpha(\ln
(m+1))P_y(\sigma_0>m^\alpha)=c_6(y),\] which matches the lower bound.

We turn now to the proof of (\ref{atp}). It is enough to show
that
\begin{equation}\label{sq}
  \lim_{n\to\infty} (\ln
n)P_y\left[\sum_{i=0}^{\sigma_0-1}V_i>n^2\right]=c_6(y).
\end{equation}

{\em Lower bound for (\ref{sq}).} By Lemma~\ref{hitn} and (\ref{survive}) ,
\begin{multline*}
  \liminf_{n\to\infty} (\ln
n)P_y\left[\sum_{i=0}^{\sigma_0-1}V_i>n^2\right]\ge \liminf_{n\to\infty}(\ln
  n)P_y\left[\sum_{i=0}^{\sigma_0-1}V_i>n^2,\tau_n<\sigma_0
  \right]\\\ge\lim_{n\to\infty}(\ln n)P_y[\tau_{
    n}<\sigma_0]\lim_{n\to\infty}P_n\left[\sum_{i=0}^{\sigma_0-1} V_i>n^2\right]
 =c_6(y).
\end{multline*}

{\em Upper bound for (\ref{sq}).} The reasoning is very similar to the
one we gave for (\ref{aet}). Fix $\alpha>1.$ Using the sequence
$m=m(n)$ such that $m^\alpha\le n<(m+1)^\alpha$ we get
\begin{equation*}
  \limsup_{n\to\infty}(\ln
  n)P_y\left[\sum_{i=0}^{\sigma_0-1}V_i>n^2\right] \le \alpha
  \limsup_{m\to \infty}\ln
  (m+1)P_y\left[\sum_{i=0}^{\sigma_0-1}V_i>m^{2\alpha} \right].
\end{equation*}
Therefore, if we show that for every $\alpha>1$
\begin{equation}\label{intb}
  \limsup_{m\to\infty}(\ln m
  )P_y\left[\sum_{i=0}^{\sigma_0-1}V_i>m^{2\alpha} \right]\le c_6(y),
\end{equation}
then letting $\alpha\to 1$ and using the lower bound we shall obtain
(\ref{sq}). Notice that 
\[ (\ln m)\, P_y\left[\sum_{i=0}^{\sigma_0-1}V_i>m^{2\alpha} ,
  \tau_m<\sigma_0\right]\le (\ln m)\, P_y(\tau_m<\sigma_0), \] which
by Lemma \ref{hitn} converges to $c_6(y)$ as $m\to\infty$. Finally,
\begin{equation*}
  P_y\left[\sum_{i=0}^{\sigma_0-1}V_i>m^{2\alpha} ,
    \tau_m>\sigma_0\right] \le
  P_y(\sigma_0>m^{2\alpha-1},\tau_m>\sigma_0)\le
  P_y\left[\sum_{i=0}^{\sigma_0-1}\mathbbm{1}_{\{V_i\le m\}}>m^{2\alpha-1}\right].
\end{equation*}
By Lemma~\ref{delay} the last expression is $o(1/\ln m)$ as
$m\to\infty$, and we get (\ref{intb}).
\end{proof}

\section{Proofs of Corollary~\ref{CrPBMPOsNeg} and
  Theorem~\ref{main_res2} }\label{four}

\subsection{Proof of Corollary \ref{CrPBMPOsNeg}}
Part (a) of Corollary~\ref{CrPBMPOsNeg} follows from the following
lemma. Observe that this lemma also covers the case $\delta=1$. This will be
needed later in the section.
\begin{lemma}
  \label{time_pos} 
  Let $\delta\in[0,1]$. Then as
  $n\to\infty$ \[\frac{A^+_n}{n}\Rightarrow
  Z\left(\frac{1+\delta}{2}, \frac{1-\delta}{2}\right),\] where we set
  $Z(1,0)\equiv 1$.
\end{lemma}

\begin{proof}
  This lemma is an easy consequence of Theorems~\ref{ThERWRecLim} and
  \ref{PrPBMPOsNeg} (for $\delta\in[0,1)$), Theorem~\ref{ThERWCrit}
  (for $\delta=1$), and the continuous mapping theorem. To unify the
  notation let \[X_{\delta,n}(\cdot):=
  \begin{cases}
    \dfrac{X_{[n\cdot]}}{\sqrt{n}},&\text{if }\delta\in[0,1);\\[3mm]
\dfrac{X_{[n\cdot]}}{b \sqrt{n}\log n},&\text{if }\delta=1;
  \end{cases}\qquad W_{1,-1}:=B^*.
  \] Define $\phi:D([0,1])\to \mathbb{R}$
  by 
  \begin{equation}
  \label{phi}
  \phi(\omega)=\int_0^1\mathbbm{1}_{[0,\infty)}(\omega(t))\,dt.
  \end{equation}
Note that the Lebesgue measure
  of the set ${\cal Z}:=\{t\in[0,1]:\ W_{\delta,-\delta}(t)=0\}$ is $0$
  $P$-a.s.. Indeed,  
  \begin{align*}
    E\int_0^1\I_{\cal Z}(t)\,dt=\int_0^1P(W_{\delta,-\delta}(t)=0)\,dt=0.
  \end{align*}
  where the last equality follows from the fact that
  $W_{-\delta, \delta}$ has a density (see \cite[Proposition 2.3 and
  Section 3.3]{CPY}). Then, if $P$ is the measure corresponding to
  $W_{\delta,-\delta}$ then by Proposition \ref{PrPhi} the map $\phi$ is
  continuous $P$-a.s. (as $P$ is supported on continuous functions) and
\[\phi(X_{\delta,n})=\frac{1}{n}
\sum_{k=0}^n\mathbbm{1}_{[0,\infty)}(X_k)=\frac{A^+_n}{n}\Rightarrow
\phi(W_{\delta,-\delta})
\overset{\mathrm{d}}{=}Z\left(\frac{1+\delta}{2},\frac{1-\delta}{2}\right).\]
The last equality follows from Theorem~\ref{PrPBMPOsNeg}(a) for
$\delta\in[0,1)$ and is trivial for $\delta=1$.
  \end{proof}
  It is enough to show the second part of
  Corollary~\ref{CrPBMPOsNeg}(b). The proof of the first part is
  similar. For every $R>0$ consider the map $\psi:D([0,\infty))\to
  D([0,R])$ defined by
  \begin{equation}
    \label{psi}
    \psi(\omega(s),\,0\le
  s<\infty)=(-\omega(T^-(s)),\,0\le T^-(s)\le R),
  \end{equation}
  where
  $T^-(s):=\inf\left\{t\ge 0:\int_0^t\I_{(-\infty,0)}(\omega(r))\,dr
    >s\right\}$.
  By Proposition \ref{PrPsi} $\psi$ is continuous $P$-a.s.\ ($P$ is
  the measure which corresponds to $W_{\delta,-\delta}$).  The desired
  statement now follows from Theorem~\ref{ThERWRecLim} and
  Theorem~\ref{PrPBMPOsNeg}(b) by the continuous mapping theorem.
  
\subsection{Heuristics and the proof of
  Theorem~\ref{main_res2}(a)} \label{4.2}

We start by introducing some additional notation which will be used
throughout the rest of Section~\ref{four}.  Denote by $d_n$ the number of
down-crossings of $X$ from $0$ to $-1$ up to time $n$ inclusively 
and
by $u_n$ the number of up-crossings of $X$ from $0$ to $1$ up to time
$n$ inclusively.  Rename the BP $V$ into $V^+$ (for right excursions) and introduce
the BP $V^-$ which corresponds to left excursions of the walk. Namely,
for $x\le 0$ let
\[F_x(0)=0,\ \ F_x(m):=\inf\left\{k\ge 1:\
  \sum_{i=1}^k\eta_x(i)=m\right\}-m,\ \ m\in\IN.\] Thus, $F_x(m)$ is
the number of ``failures'' before the $m$-th ``success'' in the sequence
$\eta_x(i)$, $i\in\IN$.  Define the process $V^-=(V_n^-)_{n\ge 0}$ which
starts with $y$ particles in generation $0$ by
\begin{equation}
  \label{V-}
  V_0^-=y,\ \
V_n^-=F_{-n}(V_{n-1}),\ \ n\in\IN.
\end{equation}
If $V_0^\pm=k$ then denote by
$\Sigma^\pm_k:=\sum_{j=0}^{\sigma_0-1}V^\pm_j$ the total progeny of
the BP $V^\pm$ over its lifetime and observe that
\begin{equation}
  \label{squeeze}
  2\Sigma^+_{u_n-1}\le A^+_n\le 2\Sigma^+_{u_n}+d_n+1 \ \ \text{and}\ \ 2\Sigma^-_{d_n-1}-d_n\le A^-_n\le 2\Sigma^-_{d_n}.
\end{equation}
To see why the first set of the above inequalities holds, note that
$A^+_n$ falls in between the total duration (including visits to $0$)
of the first $u_n-1$ and the first $u_n$ excursions to the
right. Since the number of up-crossings from one level to the next in
each excursion is equal to the number of down-crossings, by coupling
with the BP we obtain the estimates in terms of the total progeny of
the BP which starts with $u_n-1$ and $u_n$ particles
respectively. Since $A^+_n$ includes the number of visits to zero, we
have to add to the upper bound the number of visits to $0$ after which
the walker stepped to the left, i.e.\ $d_n$. An additional $1$ in the
upper bound for $A_n^+$ accounts for the possibility that $X_n\geq 0$,
in which case we have to count the up- or down- crossing in the next
step from that point.  The second set of inequalities is obtained
similarly. The only difference is that by our definition $A^-_n$ does
not include the time spent at $0$.

\medskip

\noindent{\em Informal discussion.} 
Let us explain where the uniform distribution in
Theorem~\ref{main_res2}(a) comes from. Recall that $Y$ is a half of a
squared Bessel process of dimension 2, i.e.\ the diffusion satisfying
(\ref{SqB2}), and let $\tau_x=\inf\{t\ge 0:\,Y(t)=x\}$, $x>0$.  The
uniform distribution appears naturally in the following lemma.
\begin{lemma}
  \label{elem}
Let $Y^*(t)=\max_{s\le t}Y(s)$ and $Y(0)=y>1$. 
Then \[\frac{\ln y}{\ln
  Y^*(\tau_1)} \overset{\mathrm{d}}{=}U.\] 
\end{lemma}
\begin{proof}
  It is easy to check that $\ln Y(t)$, $t\ge 0$, is a local martingale
  and so for all $R>y$
  \begin{equation}
    \label{hit}
    P_y(\tau_R<\tau_1)=\frac{\ln y}{\ln R}.
  \end{equation}
  For $x\in (0,1)$ we have
\[    P_y\left(\frac{\ln y}{\ln Y^*(\tau_1)} \le x \right)=
    P_z(Y^*(\tau_1)\ge y^{1/x })=P_y(\tau_{y^{1/x }}\le
    \tau_1)\overset{(\ref{hit})}{=}x.\qedhere \]
  \end{proof}
  The next step is to observe that for a large starting point $y$ the
  area under the path of $Y$ up to $\tau_1$ is roughly the square of
  $Y^*(\tau_1)$.
\begin{lemma}
  \label{area}
  Let $Y(0)=y>1$. Then
\begin{equation}
  \label{two}
  \frac{\ln\int_0^{\tau_1}Y(s)\,ds}{\ln Y^*(\tau_1)}\Rightarrow
  2\quad\text{as }y\to\infty.
\end{equation}
\end{lemma}
The proof of Lemma~\ref{area} is omitted as we use it only for this
informal discussion. It can be proven in the same way as
Lemma~\ref{areabp}. 
The next statement immediately follows from Lemmas~\ref{elem} and
\ref{area}.
\begin{corollary}
  Let $Y(0)=y>1$. Then
  \begin{equation}
    \label{contv}
    \frac{2\ln y}{\ln\int_0^{\tau_1}Y(s)\,ds}\Rightarrow
  U\quad\text{as }y\to\infty.
  \end{equation}
\end{corollary}
The key part of the proof of Theorem~\ref{main_res2}(a) is
the following analog of (\ref{contv}): let $V^+_0=n$, then
\begin{equation}
    \label{(ii)}
    \frac{2\ln n}
  {\ln \Sigma^+_n}\Rightarrow U\quad\text{as }n\to\infty.
  \end{equation}
  Notice that (\ref{(ii)}) could not be obtained from (\ref{contv})
  simply by the diffusion approximation, since we consider $V$ all the
  way down to the extinction time and $Y$ does not hit zero with
  probability 1. In the next subsection we prove BP versions of
  Lemmas~\ref{elem} and \ref{area} (see Lemmas~\ref{elembp} and
  \ref{areabp}) and obtain (\ref{(ii)}).

  Once we know (\ref{(ii)}), it is relatively simple to arrive at the
  conclusion of Theorem~\ref{main_res2}(a). We want to show that $\ln
  A^-_n/\ln n\Rightarrow U$. Consider the following
  chain of substitutions as $n\to\infty$:
  \begin{equation*}
    \frac{\ln A^-_n}{\ln n}\overset{(\ref{squeeze})}{\longleftrightarrow} 
    \frac {\ln \Sigma^-_{d_n}}{\ln n}\overset{\mathrm{L.\,\ref{redu}}} 
    {\longleftrightarrow} \frac{2\ln d_n}{\ln n} 
    \overset{\mathrm{L.\, \ref{ud}}} {\longleftrightarrow}\frac{2\ln u_n} 
    {\ln n} \overset{\mathrm{L.\,\ref{time_pos}}} {\longleftrightarrow} 
    \frac{2\ln u_n}{\ln A^+_n} \overset{(\ref{squeeze})} 
    {\longleftrightarrow} \frac{2\ln u_n}{\ln \Sigma^+_{u_n}}
    \longleftrightarrow\frac{2\ln n}{\ln \Sigma^+_n},
  \end{equation*}
  where the last ratio converges to $U$ by (\ref{(ii)}).  The actual
  proof combines the last three steps into a single argument.  Below
  we state Lemmas~\ref{redu} and \ref{ud} mentioned above, and use
  them together with (\ref{(ii)}) to derive
  Theorem~\ref{main_res2}(a). The proofs of Lemmas~\ref{redu} and
  \ref{ud} are postponed until Section~\ref{456}.

\begin{lemma}
  \label{redu} For every $\nu>0$, $x\in[0,1]$, and all
  sufficiently large $n$
  \begin{equation*}
    P_0\left(\frac{2\ln d_n}{\ln n}\le x-\nu\right)-\nu\le 
P_0\left(\frac{\ln \Sigma^-_{d_n}}{\ln n}\le x\right)\le 
P_0\left(\frac{2\ln d_n}{\ln n}\le x+\nu\right)+\nu
  \end{equation*}
\end{lemma}

\begin{lemma}
  \label{ud} The following statements hold with probability 1 as $n\to\infty$:
  \begin{align}
    u_n&\to\infty;\ \ d_n\to\infty;\label{ab}\\ &\frac{d_n}{u_n}\to 1.\label{du}
  \end{align}
\end{lemma}
  
\medskip

\begin{proof}[Proof of Theorem~\ref{main_res2}(a)]
  By (\ref{squeeze}), Lemma~\ref{redu}, and Lemma~\ref{ud} it is
  enough to show that $2(\ln u_n)/(\ln n)\Rightarrow U$ as $n\to\infty$.

  Let $x\in(0,1)$. Fix an arbitrary $\nu>0$ and 
  $\epsilon\in(0,1/2)$. Then
\begin{align*}
  P\left(\frac{2\ln u_n}{\ln n}\le x\right) &\le P\left(
    2\Sigma^+_{\floor{n^{x/2}}}\ge
    (1-2\epsilon)n\right)+P\left(\frac{2\ln u_n}{\ln n}\le
    x,\,2\Sigma^+_{\floor{n^{x/2}}} <(1-2\epsilon)n\right)\\&=
  P\left(2\Sigma^+_{\floor{n^{x/2}}}\ge
    (1-2\epsilon)n\right)+P\left(u_n\le \floor{n^{x/2}},\,
    2\Sigma^+_{\floor{n^{x/2}}} <(1-2\epsilon)n\right).\\
\end{align*}
Note that by Lemma~\ref{time_pos} with probability at least $1-\nu/2$
for all large $n$ we have that
\begin{equation}
\label{A-u}
(1-\epsilon)n+1\le A^+_n\overset{(\ref{squeeze})}{\le} 2\Sigma^+_{u_n}+d_n+1.
\end{equation}
Moreover, on the
set $\{u_n\le \floor{n^{x/2}}\}$ we have by coupling that
\begin{equation}
\label{SIncr}
\Sigma^+_{u_n}\le \Sigma^+_{\floor{n^{x/2}}}.
\end{equation}
Inequalities \eqref{A-u} and \eqref{SIncr} imply that 
\[2\Sigma^+_{\floor{n^{x/2}}}\ge (1-\epsilon)n-d_n\ge
n\left(1-\epsilon-\frac{d_n}{n}\right).\]
Since $d_n/n\le A^-_n/n\to 0$ in probability by Lemma~\ref{time_pos},
we can conclude that for all sufficiently large $n$
\[P\left(u_n\le \floor{n^{x/2}}, 2\Sigma^+_{\floor{n^{x/2}}}
  <(1-2\epsilon)n\right)\le \nu.\]
Hence, for all sufficiently large $n$
\begin{align*}
  P\left(\frac{2\ln u_n}{\ln n}\le x\right)&\le
  P\left(2\Sigma^+_{\floor{n^{x/2}}}\ge (1-2\epsilon)n\right)+\nu \\&\le
  P\left(\frac{2\ln\floor{n^{x/2}}}{\ln \Sigma^+_{\floor{n^{x/2}}}}\le
    x+\nu\right)+\nu\overset{(\ref{(ii)})}{\le} x+3\nu.
\end{align*}

Towards a lower bound, observe that by coupling
$\{2\Sigma^+_z > n\}\subset \{u_n\le z\}$ for all $z\in\IN$. Using
this fact, Lemma~\ref{ud}, and (\ref{(ii)}) we get for all
sufficiently large $n$ that
\begin{align*}
  P\left(\frac{2\ln u_n}{\ln n}\le x\right) &\ge P\left(u_n\le
    \floor{n^{x/2}},\,2\Sigma^+_{\floor{n^{x/2}}} > n\right)\\&=
  P\left(2\Sigma^+_{\floor{n^{x/2}}}> n\right)\ge
  P\left(\frac{2\ln\floor{n^{x/2}}}{\ln
      \Sigma^+_{\floor{n^{x/2}}}}\le x-\nu\right)\ge x-3\nu. \qedhere 
\end{align*}
\end{proof}

\subsection{The lifetime maximum and progeny of a critical BP}
In this subsection we prove (\ref{(ii)}). It is an immediate consequence of
the following two lemmas.
\begin{lemma}
  \label{elembp} 
Let $V_0=n>1$. Then \[\frac{\ln n}{\ln
    \max_{j<\sigma_0}V_j} \Rightarrow U\quad\text{as }n\to\infty.\]
\end{lemma}

\begin{lemma}
  \label{areabp} Let $V_0=n>1$. Then \[\frac{\ln
    \sum_{j=0}^{\sigma_0-1}V_j}{\ln \max_{j<\sigma_0}V_j} \Rightarrow
  2\quad\text{as }n\to\infty.\]
\end{lemma}

\begin{proof}[Proof of Lemma~\ref{elembp}]
  For every $x\in(0,1)$ 
  \[P_n(\max_{j<\sigma_0}V_j\ge n^{1/x})=P_n(\tau_{n^{1/x}}<\sigma_0).\] 
  The proof will be complete
  if we can show that the last probability converges to $x$ as
  $n\to\infty$. Fix a large enough $y\in\mathbb{N}$ to satisfy the
  conditions of Lemma \ref{main}. Then
  \begin{align*}
    P_n(\tau_{n^{1/x}}<\sigma_0)&= P_n(\tau_{n^{1/x}}<\sigma_y)+
    P_n(\tau_{n^{1/x}}<\sigma_0\,|\,\tau_{n^{1/x}}>\sigma_y)
    P_n(\tau_{n^{1/x}}>\sigma_y)\\&\le P_n(\tau_{n^{1/x}}<\sigma_y)+
    P_y(\tau_{n^{1/x}}<\sigma_0).
  \end{align*}
  By Lemma \ref{main}
  the first term in the right-hand side of the above inequality
  is bounded above by  $\ceil{\log_2 n}/\floor{x^{-1}\log_2 n}$ which converges to $x$
  as $n\to\infty$. By Corollary~\ref{uub}
\[P_y(\tau_{n^{1/x}}<\sigma_0)\le \frac{c_{10}(y)}{\floor{x^{-1}\ln n}}\to 0\ \ \text{as }n\to\infty.\] 
The lower bound is even easier. By  Remark~\ref{r}
\[\liminf_{n\to\infty}P_n(\tau_{n^{1/x}}<\sigma_0)\ge
  \liminf_{n\to\infty}P_n(\tau_{n^{1/x}}<\sigma_y)\ge x.\qedhere\]
\end{proof}

\begin{proof}[Proof of Lemma~\ref{areabp}] Fix $\epsilon\in (0,1)$,
  let $k_0=\floor{\log_2 n}$. 

  {\em Lower ``tail''.} To get a bound on the probability that the
  ratio in Lemma~\ref{areabp} is not less than $2-\epsilon$, we split
  the path space of the process $V$ according to its lifetime
  maximum. On each event $\tau_{2^k}<\sigma_0<\tau_{2^{k+1}}$, $k\ge
  k_0$, we shall take into account only the values of $V$ from the
  time $\tau_{2^k}$ up until the time
  $\sigma'_{2^{k-1}}:=\inf\{i>\tau_{2^k}:\,V_i\le 2^{k-1}\}$. On the
  time interval $\{i\in\mathbb{N}:\,\tau_{2^k}\le
  i<\sigma'_{2^{k-1}}\}$ the process $V$ stays above $2^{k-1}$ and
  below $2^{k+1}$. Thus,
\begin{align*}
  P_n&\left(\sum_{i=0}^{\sigma_0-1}V_i\le \left(\max_{i<\sigma_0}
      V_i\right)^{2-\epsilon}\right) \le \sum_{k=k_0}^\infty P_n
  \left(\sum_{i=\tau_{2^k}}^{\sigma'_{2^{k-1}}-1}V_i \le
    2^{(k+1)(2-\epsilon)},\tau_{2^k}<\sigma_0<\tau_{2^{k+1}}\right)\\
  &\le \sum_{k=k_0}^\infty P_n \left(2^{k-1}(\sigma'_{2^{k-1}}-\tau_{2^k})\le
    2^{(k+1)(2-\epsilon)},\tau_{2^k}<\sigma_0<\tau_{2^{k+1}}\right)\\
  &\le \sum_{k=k_0}^\infty
  E_n\left(\mathbbm{1}_{\{\tau_{2^k}<\sigma_0\}}P_n
    \left(\sigma'_{2^{k-1}}-\tau_{2^k}\le
      2^{k(1-\epsilon)+3},\sigma_0<\tau_{2^{k+1}}\Big|{\cal
        F}_{\tau_{2^k}}\right)\right)\\
  &\le \sum_{k=k_0}^\infty
  P_n(\tau_{2^k}<\sigma_0)P_{2^k}
    \left(\sigma'_{2^{k-1}}\le
      2^{k(1-\epsilon)+3},\sigma_0<\tau_{2^{k+1}}\right)= \sum_{k=k_0}^\infty A_{n,k}B_k,
\end{align*}
where $A_{n,k}=P_n(\tau_{2^k}<\sigma_0)$ and $B_k=P_{2^k}
\left(\sigma'_{2^{k-1}}\le
  2^{k(1-\epsilon)+3},\sigma_0<\tau_{2^{k+1}}\right)$, which we estimate separately.

Let $\ell_0<k_0$ be fixed as in  Lemma \ref{main}, 
$k_0$ be sufficiently large, and $k\ge k_0+2$ (for
$k=k_0,k_0+1$ we shall use the trivial bound $A_{n,k}\le 1$). Then
\begin{align*}
  A_{n,k}&=P_n(\tau_{2^k}<\sigma_{2^\ell})+P_n(\tau_{2^k}<\sigma_0\,|\sigma_{2^\ell}<\tau_{2^k})
  P_n(\sigma_{2^\ell}<\tau_{2^k})\\
  &\leq P_{2^{k_0+1}}(\tau_{2^k}<\sigma_{2^\ell})+P_{2^\ell}(\tau_{2^k}<\sigma_0)
  \overset{\text{L.\,\ref{main},\ L.\,\ref{hitn}}}{\le}
    \frac{k_0+1}{k}+\frac{C(\ell)}{k}.
\end{align*}
Fix an arbitrary $\nu>0$. If $k_0$ is large enough then for all $k\ge k_0$
\begin{align*}
  B_k&\le P_{2^k}(\sigma_{2^{k-1}}\le
  2^{k(1-\epsilon)+3},\sigma_0<\tau_{2^{k+1}},V_{\sigma_{2^{k-1}}}\ge
  2^{k-2})+
  P_{2^k}(\sigma_{2^{k-1}}<\tau_{2^{k+1}},V_{\sigma_{2^{k-1}}}<
  2^{k-2})\\ &\overset{{\rm (OS)}}{\le
  }E_{2^k}\left(\mathbbm{1}_{\{\sigma_{2^{k-1}}\le
      2^{k(1-\epsilon)+3}\}}P_{2^k}(\sigma_0<\tau_{2^{k+1}},V_{\sigma_{2^{k-1}}}\ge
    2^{k-2}|\,{\cal F}_{2^{k-1}})\right)+c_7\exp(-c_92^k)\\ &\le
  P_{2^k}(\sigma_{2^{k-1}}\le
  2^{k(1-\epsilon)+3})P_{2^{k-2}}(\sigma_{2^{\ell_0}}<\tau_{2^{k+1}})+
  c_7\exp(-c_92^k)\\ &\overset{\text{(DA),\,L.~\ref{main}}}{\le}
  \frac{\nu}{k-\ell_0}+ c_7\exp(-c_92^k).
\end{align*}
Substituting the estimates for $A_{n,k}$ and $B_k$ we conclude that
for all sufficiently large $n$ 
\[P_n\left(\sum_{i=0}^{\sigma_0-1}V_i\le \left(\max_{i<\sigma_0}
    V_i\right)^{2-\epsilon}\right)\le
3\nu+\nu(k_0+1+C(\ell_0))\sum_{k=k_0+2}^\infty
\frac{1}{k(k-\ell_0)}<C_1(\ell_0)\nu.\]

{\em Upper ``tail''.} To get a bound on the probability that the ratio
in Lemma~\ref{areabp} is at least $2+\epsilon$ we
let \[O_j:=\frac{1}{2^j}\sum_{i=0}^{\sigma_0-1}\mathbbm{1}_{\{2^j\le
  V_i<2^{j+1}\}},\ \ k^*:=\floor{\log_2 \max_{i<\sigma_0}V_i},\ \
m_k=\floor{2^{\epsilon k-2}},\] and use a crude ``union bound'':
\begin{align}\label{ut}
  P_n\left(\sum_{i=0}^{\sigma_0-1}V_i\ge \left(\max_{i<\sigma_0}
      V_i\right)^{2+\epsilon}\right)
  &\le P_n\left(\sum_{j=0}^{k^*}2^{2j+1}O_j\ge 2^{k^*(2+\epsilon)}\right)
    \nonumber\\ &\le P_n\left(\max_{0\le j\le k^*}O_j\ge m_{k^*}\right)\le
    \sum_{k=k_0}^\infty\sum_{j=0}^k P_n(O_j\ge m_k).
\end{align}
To estimate the (rescaled) time $O_j$ which the process $V$ spends in
the interval $[2^j,2^{j+1})$, $j\ge 0$, we define \[\rho_0^{(j)}:=\inf\{i\ge
0:\,V_i\in[2^j,2^{j+1})\},\quad \rho_m^{(j)}:=\inf\{i\ge
\rho_{m-1}^{(j)}+2^j:\,V_i\in[2^j,2^{j+1})\},\
m\in\mathbb{N}.\] Then by the strong Markov property
\begin{multline*}
  P_n(O_j\ge m_k)\le P_n(\rho_{m_k}^{(j)}<\sigma_0)\le
  P_n(\rho_{m_k}^{(j)}<\sigma_0\,|\,
  \rho_{m_k-1}^{(j)}<\sigma_0)P_n(\rho_{m_k-1}^{(j)}<\sigma_0)\\\le
  \left(\max_{2^j\le x<2^{j+1}}
    P_x(\rho_1^{(j)}<\sigma_0)\right)^{m_k}P_n(\rho_0^{(j)}<\sigma_0).
\end{multline*}
We notice that by (DA) there is a $c>0$ such that
$P_{2^{j+1}}(\sigma_{2^{j-1}}<2^{j-1})>c$ for all $j\ge 2$, and
choosing $\ell_0$ as in Lemma \ref{main} we get that if 
$(\ell_0+1)\wedge c_{10}(2^{\ell_0})<j\le k$ 
where $c_{10}$ is from Corollary~\ref{uub} then
\begin{align*}
  \max&_{2^j\le x<2^{j+1}} P_x(\rho_1^{(j)}<\sigma_0)\le
  1-\min_{2^j\le x<2^{j+1}} P_x(\rho_1^{(j)}>\sigma_0,
  \sigma_{2^{j-1}}<2^{j-1})\\ &\le 1-\min_{2^j\le x<2^{j+1}}
  P_x(\rho_1^{(j)}>\sigma_0\,|\,
  \sigma_{2^{j-1}}<2^{j-1})P_{2^{j+1}}(\sigma_{2^{j-1}}<2^{j-1})\\
  &\le 1-cP_{2^{j-1}} (\sigma_0<\tau_{2^j})\le 1-cP_{2^{j-1}}
  (\sigma_0<\tau_{2^j}, \sigma_{2^{\ell_0}}<\tau_{2^j})\\
  &\le 1-cP_{2^{\ell_0}}(\sigma_0<\tau_{2^j})P_{2^{j-1}}
  (\sigma_{2^{\ell_0}}<\tau_{2^j}) \overset{\rm Cor.\ref{uub}}{\underset{
      \rm Rem.\ref{r}}{\le}}1-c\left(1-
    \frac{c_{10}(\ell_0)}{j}\right)\frac{1}{j}\le 1-\frac{c'}{k}.
\end{align*}
Choosing $k_0$ large enough we can also ensure that for all $k\ge
k_0$ \[\max_{0\le j\le (\ell_0+1)\wedge c_{10}(\ell_0)}\max_{2^j\le x<2^{j+1}}
P_x(\rho_1^{(j)}<\sigma_0)\le 1-c'/k. \] Substituting these estimates in
(\ref{ut}) we conclude that 
\[P_n\left(\sum_{i=0}^{\sigma_0-1}V_i\ge
  \left(\max_{i<\sigma_0} V_i\right)^{2+\epsilon}\right)\le
\sum_{k=k_0}^\infty(k+1)\left(1-\frac{c'}{k}\right)^{m_k}\to 0\ \text{
  as } n\to\infty. \qedhere\]
\end{proof}

\subsection{Proofs of Lemmas~\ref{redu} and \ref{ud}}\label{456}
We shall need the following result.
  \begin{lemma}[(4.4) from Theorem~4.1 of \cite{KZ14}]
    \label{km}
   Let 
    $(Y^-(t)),\ t\ge 0$, be the solution
    of \[dY^-(t)=-dt+\sqrt{2Y^-(t)}\,dB(t),\quad Y^-(0)=1,\quad t\in[0,
    \tau_0].\] Then for every $h>0$ 
    \begin{equation}
   \label{EqInt} 
   \lim_{n\to\infty}P^{V^-}\left(\Sigma^-_n
         > hn^2\right)=
   P_1^{Y^-}\left(\int_0^{\tau_0}Y^-(s)\,ds> h\right).
    \end{equation}
  \end{lemma}

\begin{proof}[Proof of Lemma~\ref{redu}]
 Upper bound:
  \begin{align*}
    P_0(\Sigma^-_{d_n}\le n^x)&\le P_0(d_n\le n^{x/2}\ln n) +P_0(\Sigma^-_{d_n}\le
    n^x,d_n>n^{x/2}\ln n)\\&\le P_0(d_n\le n^{x/2}\ln
    n)+P_{\floor{n^{x/2}\ln n}}^V(\Sigma^-_{\floor{n^{x/2}\ln n}}\le n^x)\\
    &\overset{\eqref{EqInt}}{\le} P_0(d_n\le n^{x/2}\ln
    n)+\nu.
  \end{align*}
Lower bound:
\begin{align*}
  P_0(\Sigma^-_{d_n}\le n^x)&\ge P_0(d_n\le n^{x/2}/\ln
  n)-P_0(\Sigma^-_{d_n}>n^x,d_n\le n^{x/2}/\ln n)\\&\ge P_0(d_n\le
  n^{x/2}/\ln n)-P^V_{\floor{n^{x/2}/\ln n}}(2\Sigma^-_{\ceil{n^{x/2}\ln
      n}}>n^x)\\&\ge P_0(d_n\le n^{x/2}/\ln n)-\nu. \qedhere
\end{align*}
\end{proof}

\begin{proof}[Proof of Lemma~\ref{ud}]
  Let $L_n$ be the number of visits of $X$ to $0$ up to time $n$
  inclusively. Since $0\le L_n-(u_n+d_n)\le 1$ and the ERW with
  $\delta=1$ is recurrent, we have that $L_n-u_n-1\le d_n\le L_n-u_n$,
  $L_n\to\infty$ a.s., and both (\ref{ab}) and (\ref{du}) would follow
  if we show that
  \begin{equation}
    \label{dl}
    \frac{u_n}{L_n}\to \frac{1}{2}\ \ \text{as $n\to\infty$ a.s..} 
  \end{equation}
Notice that 
  \begin{equation}\label{est}
    \frac{\sum_{i=M+1}^{L_n}\eta_0(i)}{L_n}\le \frac{u_n}{L_n}\le 
\frac{M+\sum_{i=M+1}^{L_n}\eta_0(i)}{L_n}.
\end{equation}
As $L_n\to\infty$ a.s.\ as $n\to\infty$, the rightmost and leftmost
ratios in (\ref{est}) a.s.\ converge to $1/2$ by the strong law of
large numbers for Bernoulli trials.
\end{proof}

\subsection{Proof of Theorem \ref{main_res2}(b),(c)}
\begin{proof}[Proof of Theorem \ref{main_res2}(b)] 
Let $\tX_n$ denote the excited random walk in the cookie environment
obtained by removing all cookies from the positive semi-axis. The same
proof as for \cite[Theorem 1.1]{DK12} shows that
\begin{equation}
  \label{x0}
  \frac{\tX_{\floor{n\,\cdot}}}{\sqrt{n}}\overset{J_1}{\Rightarrow} W_{0,-1}. 
\end{equation}
Namely, we write $\tX_n=\tB_n+\tC_n$, where $\tB_0=\tC_0=0$ and
\[ \tB_{n+1}-\tB_n=\tX_{n+1}-\tX_n, \quad \tC_{n+1}-\tC_n=0 \]
if $\tX$ visited $\tX_n$ at least $M$ times before time $n$ and
\[ \tB_{n+1}-\tB_n=0, \quad 
\tC_{n+1}-\tC_n=\tX_{n+1}-\tX_n \]
otherwise. Then we can show that
\[ \left(\frac{\tB_{\floor{n\,\cdot}}}{\sqrt{n}},
  \frac{\tC_{\floor{n\,\cdot}}}{\sqrt{n}}\right)\overset{J_1}{\Rightarrow}
\left(B(\cdot), -\min_{s\leq \cdot} B(s)\right), \] and obtain
(\ref{x0}). We refer to \cite{DK12} for full details.
Since there is an obvious coupling such that $\tX_{T^-_k}=X_{T^-_k}$, $k\ge 0$,
the result follows from Theorem \ref{PrPBMPOsNeg}(b) and the
continuity of the map $\psi$ defined in (\ref{psi}).
\end{proof}

\begin{proof}[Proof of Theorem \ref{main_res2}(c)] This result admits
  the same proof as the one for Corollary~\ref{CrPBMPOsNeg}(b) but,
  since $A^+_n/n\to 1$ for $\delta=1$, we can give a simpler derivation.
  
  Without loss of generality we show the convergence on
  $D([0,1])$. Write
\[ X_{T^+_m}=X_m+(X_{T^+_m}-X_m) . \]
By Lemma \ref{time_pos} for each $\epsilon, \nu>0$ and all large $n$
\[ P\left(\max_{m\leq n} (T^+_m-m)\ge \epsilon n\right)\leq \nu. \]
 On the other hand, given arbitrary positive
$\epsilon$ and $\nu$ we can choose $\lambda>0$ so that
\[ P\left(\sup_{0\le s\le t\le s+\lambda\le 1+\lambda}
  (B^*(t)-B^*(s))>\epsilon\right)\leq \nu. \] The above inequalities
and Theorem \ref{ThERWCrit} imply that for any fixed $\epsilon,\nu>0$ and all sufficiently large $n$
\[ P\left(\max_{m\le n}|X_{T^+_m}-X_m|>\epsilon \sqrt{n}\ln n
\right)<\nu. \] 
Theorem~\ref{ThERWCrit} and the ``convergence together'' theorem
\cite[Theorem 3.1]{B99} imply the desired result. 
\end{proof}

\appendix

\section{Proofs of Lemmas~\ref{hitn} - \ref{one}}
The following lemma plays an important role in proofs of
Lemmas~\ref{hitn} and \ref{delay}.
\begin{lemma}[Main lemma] \label{main}
Let  
\[h^\pm(n):=n\pm\frac1n\qquad\mbox{for all $n\in\mathbb{N}$}.
\]
Then there is $\ell_0\in\mathbb{N}$ such that if
$\ell,m,u,x\in\mathbb{N}$ satisfy $\ell_0\le \ell<m<u$ and $|x-2^m|\le
2^{2m/3}$ then
\begin{equation}\label{ip}
\frac{h^-(u)-h^-(m)}{h^-(u)-h^-(\ell)}\le P_x[\sigma_{2^\ell}<\tau_{2^u}]\le
\frac{h^+(u)-h^+(m)}{h^+(u)-h^+(\ell)}.
\end{equation}
\end{lemma}
\begin{remark}\label{r}
  A little algebra shows that the lower bound is at least $1-m/u$.
\end{remark} 
The proof of Lemma~\ref{main} is the same as that of Lemma
5.3 in \cite{KM} where we take $a=2$, $h^\pm_a(n)=n\pm 1/n$, and use
the following result instead of \cite[Lemma~5.2]{KM}.
\begin{lemma}
  Consider the process $V$ with $|V_0-2^n|\le 2^{2n/3}$ and let
  $T:=\inf\{k\ge 0:\, V_k\not\in(2^{n-1},2^{n+1})\}$. Then for all
  sufficiently large $n$
  \begin{align}
    \label{strip} P(\mathrm{dist}(V_T,(2^{n-1},2^{n+1}))\ge
    2^{2(n-1)/3})&\le \exp(-2^{n/4});\\ \label{down} \left|P(V_T\le
      2^{n-1})-\frac12\right| &\le 2^{-n/4}.
  \end{align}
\end{lemma}
The proof of the above lemma repeats the one of \cite[Lemma 5.2]{KM} where
we use our process $V$, set $a=2$, and $s(x)=\ln x$ on $(3^{-1},
3)$. 

\begin{corollary}\label{uub}
  For every $y\in \mathbb{N}$ there is a constant $c_{10}(y)$ such that
  for every $n\in\mathbb{N}$
  \[(\ln n)P_y(\sigma_0>\tau_n)\le c_{10}(y).\]
\end{corollary}
The proof of this corollary is the same as that of (5.4) in \cite{KM}
and uses Lemma~\ref{main} instead of Lemma 5.3 of \cite{KM}.

\begin{corollary} \label{rc} Let $\delta=1$. Then
  $P^V_y(\sigma_0^V<\infty)=1$ for every $y\in\mathbb{N}$.
  \end{corollary}
  \begin{proof}
By Corollary~\ref{uub} and the fact that
    $P_y(\sigma_0=\infty,\tau_n=\infty)=0$ for $n>y$,
\[      P_y(\sigma_0=\infty)=P_y(\sigma_0=\infty,\tau_n<\infty)
\le P_y(\sigma_0>\tau_n)\le \frac{c_{10}(y)}{\ln n}\to 0\quad \text{as }n\to\infty. \qedhere \]
  \end{proof}
  \begin{remark}\label{rt}
    Corollary~\ref{rc}, the first statement of
    \cite[Corollary~7.9]{KZ14}, and symmetry imply that ERW with
    $|\delta|=1$ is recurrent without using any results from the
    literature on branching processes. A direct proof of recurrence
    and transience results for $|\delta|\ne 1$ was obtained in
    \cite[Corollary 7.10]{KZ14}.
  \end{remark}

\begin{proof}[Proof of Lemma~\ref{hitn}] For every $n>2$ there is an
  $m\in\mathbb{N}$ such that $2^m\le n<2^{m+1}$ and for this $m$ \[(\ln 2^m)
  P_y(\sigma_0>\tau_{2^{m+1}})\le (\ln n) P_y(\sigma_0>\tau_n)\le (\ln
  2^{m+1}) P_y(\sigma_0>\tau_{2^m}).\] If we can show the existence of
  \begin{equation}
    \label{g}
    g(y):=\lim_{m\to\infty} mP_y(\sigma_0>\tau_{2^m})\in(0,\infty),
  \end{equation}
  then we get
  \begin{align*}
    &\limsup_{n\to\infty}(\ln n)P_y(\sigma_0>\tau_n)\le \ln 2\lim_{m\to\infty}(m+1) P_y( \sigma_0>\tau_{2^m})=(\ln 2) g(y)\\
    &=\ln 2\lim_{m\to\infty}mP_y(
    \sigma_0>\tau_{2^{m+1}})\le \liminf_{n\to\infty}(\ln
    n)P_y(\sigma_0>\tau_n),
  \end{align*}
  and the desired statement follows. Therefore, we shall show
  (\ref{g}).  Let $\ell=(\floor{\log_2 y}+1)\vee \ell_0$, where
  $\ell_0$ is the same as in Lemma~\ref{main}. Then
  \begin{multline*}
    mP_y(\sigma_0>\tau_{2^m})=m\left[\prod_{j=\ell+1}^m
    P_y(\sigma_0>\tau_{2^j}\,|\,\sigma_0>\tau_{2^{j-1}})\right]P_y(\sigma_0>\tau_{2^\ell})\\
    =\ell
    P_y(\sigma_0>\tau_{2^\ell})\left[\prod_{j=\ell+1}^m \frac{j}{j-1}
    P_y(\sigma_0>\tau_{2^j}\,|\,\sigma_0>\tau_{2^{j-1}})\right].
  \end{multline*}
  We need to prove that the last product converges. For this it is
  sufficient to show that 
  \[\sum_{j=\ell+1}^\infty\left|\frac{j}{j-1}P_y(\sigma_0>\tau_{2^j}\,|
    \,\sigma_0>\tau_{2^{j-1}})-1\right|<\infty.\] Lemma~\ref{main}
  and Corollary~\ref{uub} allow us to obtain the necessary estimates.
  \begin{align*}
    \frac{j}{j-1}&P_y(\sigma_0>\tau_{2^j}\,|
    \,\sigma_0>\tau_{2^{j-1}})-1\ge
    \frac{j}{j-1}P_{2^{j-1}}(\sigma_0>\tau_{2^j})-1\\&\ge
    \frac{j}{j-1}P_{2^{j-1}}(\sigma_{2^{\ell}}>\tau_{2^j})-1 
    \stackrel{(\ref{ip})}{\ge}
    \frac{j}{j-1}\,\frac{j-1+\frac{1}{j-1}-\ell-\frac1{\ell}}{j+\frac{1}{j}-\ell-\frac1{\ell}}
    -1
    \\&=\frac{\frac{2j-1}{j^2(j-1)^2}-\frac{\ell}{j(j-1)}-\frac{1}{\ell
        j(j-1)}}{1+\frac{1}{j^2}-\frac{\ell}{j}-\frac{1}{\ell j}}\ge
    \frac{2}{j^2(j-1)}-\frac{\ell}{j(j-1)} -\frac{1}{\ell j(j-1)}.
  \end{align*}
  The right hand side of the above expression is a term of an
  absolutely convergent series.

Set $x:=2^{j-1}+2^{2(j-1)/3}$. Then \[\frac{j}{j-1}P_y(\sigma_0>\tau_{2^j}\,|
   \,\sigma_0>\tau_{2^{j-1}})\le
   \frac{j}{j-1}\left(P_x(\sigma_0>\tau_{2^j})+P_y(V_{\tau_{2^{j-1}}}>x\,|\,
     \sigma_0>\tau_{2^{j-1}})\right).\] By (OS) the last term decays exponentially fast in $j$, and we shall concentrate on the first term in the right hand side of the above inequality.
 For all sufficiently large $j$
\begin{align*}
  \frac{j}{j-1}\,&P_x(\sigma_0>\tau_{2^j})-1\\&\le \frac{j}{j-1}\,
  P_x(\sigma_{2^\ell}>\tau_{2^j})-1+\frac{j}{j-1}\,
  P_x(\sigma_0>\tau_{2^j}\,|\,\sigma_{2^\ell}<\tau_{2^j})
  P_x(\sigma_{2^\ell}<\tau_{2^j})\\&\stackrel{(\ref{ip})}{\le}
  \frac{j}{j-1}\,\frac{j-1}{j}-1+\frac{j}{j-1}\,P_{2^\ell}(\sigma_0>\tau_{2^j})
  \frac{j+\frac1j-(j-1)-\frac1{j-1}}{j+\frac1j-\ell-\frac1{\ell}}
  \\&\le \frac{j}{(j-1)(j-\ell-1)}P_{2^\ell}(\sigma_0>\tau_{2^j})
  \stackrel{\mathrm{Cor.\,\ref{uub}}}{\le}\frac{C(\ell)}{(j-1)(j-\ell-1)}.
 \end{align*}
Again the last expression is a term of a convergent series, and we are done.
\end{proof}

The proof of Lemma~\ref{delay} depends on an estimate of the
time the branching process $V$ spends in an interval $[x,2x)$ before
extinction.
\begin{lemma}\label{LmInt}
  For every $\alpha>1$ there is a constant $c_{11}(\alpha)\in(0,1)$ such
  that for all $k,x,y\in\mathbb{N}$ \[P_y\left(\sum_{j=0}^{\sigma_0-1}
    \mathbbm{1}_{[x,2x)}(V_j)>2kx^\alpha\right)\le
  P_y(\rho_0<\sigma_0)(1-c_{11}(\alpha))^k, \] where $\rho_0:=\inf\{j\ge
  0:\,V_j\in[x,2x)\}$;
\end{lemma}
\begin{proof}
  The proof is very similar to the one of Proposition 6.1 in
  \cite{KM}. There are two differences. First, everywhere in the proof of
  Proposition 6.1 the statement (ii) should be replaced with the
  following: there is a constant $c=c(\alpha)>0$ such that for all
  $x\in\mathbb{N}$
  \begin{equation}
    \label{ext}
    P_{x/2}(\sigma_0<\tau_{x^\alpha})>c.
  \end{equation}
  Second, the stopping times $\rho_j$, $j\in\mathbb{N}$, should be
  defined as follows: $\rho_0$ was defined
  above, \[\rho_j=\inf\{r\ge\rho_{j-1}+2x^\alpha:\,
  V_r\in[x,2x)\},\quad j\ge 1.\] Below we show (\ref{ext}). The rest
  of the proof is the same as that of \cite[Proposition
  6.1]{KM}. 

  To prove (\ref{ext}) we fix a large $y\in\mathbb{N}$ and observe
  that by Corollary~\ref{uub} and Remark~\ref{r} for all $x>2y+1$
  \begin{align*}
    P_{x/2}(\sigma_0<\tau_{x^\alpha})&=P_{x/2}(\sigma_0<\tau_{x^\alpha}\,|\,
    \sigma_y<\tau_{x^\alpha})P_{x/2}(\sigma_y<\tau_{x^\alpha})\\&\ge
    P_y(\sigma_0<\tau_{x^\alpha})P_{x/2}(\sigma_y<\tau_{x^\alpha})=(1-P_y(\sigma_0>\tau_{x^\alpha}))
    P_{x/2}(\sigma_y<\tau_{x^\alpha})\\&\ge
    \left(1-\frac{c_{10}(y)}{\alpha\ln x}\right)\left(1-\frac{\ln(x/2)}
      {\alpha \ln x}\right)\ge \left(1-\frac{c_{10}(y)}{\alpha\ln
        x}\right)\frac{\alpha-1}{\alpha}>c>0.
  \end{align*}
  Adjusting the constant $c$ if necessary we obtain (\ref{ext}) for
  all $x\in\mathbb{N}$.
\end{proof}
\begin{proof}[Proof of Lemma~\ref{delay}]
  For every $n\in\mathbb{N}$ let $k\in\mathbb{N}$ be such that
  $2^{k-1}\le n<2^k$. We can always write $\alpha$ as
  $\alpha'+\lambda$ where $\alpha'>1$ and $\lambda>0$. Then
  \begin{align*}
    P_y&\left(\sum_{j=0}^{\sigma_0-1}\mathbbm{1}_{\{V_j\le
        n\}}>n^\alpha\right)\le
    P_y\left(\sum_{j=0}^{\sigma_0-1}\mathbbm{1}_{\{V_j<2^k\}}>2^{\alpha(k-1)}\right)\\&\le
    P_y\left(\sum_{j=0}^{\sigma_0-1}\sum_{i=1}^k\mathbbm{1}_{[2^{i-1},2^i)}
      (V_j) >2^{\lambda(k-1)}(1-2^{-\alpha'})\sum_{i=1}^k
      2^{\alpha'(i-1)}\right)\\&\le
    \sum_{i=1}^kP_y\left(\sum_{j=0}^{\sigma_0-1}\mathbbm{1}_{[2^{i-1},2^i)}
      (V_j)>2^{\lambda(k-1)}(1-2^{-\alpha'})2^{\alpha'(i-1)}\right)\\
    &\stackrel{\mathrm{Lem.}\
      \ref{LmInt}}{\le}k(1-c_{11}(\alpha'))^{\floor{2^{\lambda(k-1)-1}(1-2^{-\alpha'})}}.
  \end{align*}
  Multiplying by $\ln n$ which is less than $k\ln 2$ we get that as $n\to\infty$
\begin{equation*}
  (\ln n)\,P_y\left(\sum_{j=0}^{\sigma_0-1}\mathbbm{1}_{\{V_j\le
        n\}}>n^\alpha\right)\le (\ln 2) k^2(1-c_{11}(\alpha'))^{\floor{2^{\lambda(k-1)-1}(1-2^{-\alpha'})}}\to 0.\qedhere
\end{equation*}
\end{proof}

Before we turn to the proof of Lemma~\ref{one} we present its
continuous space-time version.
\begin{lemma}\label{ones}
  Let $Y$ be the diffusion defined by (\ref{SqB2}) which starts at $1$ and $\tau_\epsilon:=\inf\{t\ge 0:\, Y(t)=\epsilon\}$. Then for every $h>0$
  \begin{align}
    \label{time} &\lim_{\epsilon\to
      0}P_1^Y(\tau_\epsilon>h)=1;\\ \label{int} &\lim_{\epsilon\to 0}
    P_1^Y\left(\int_0^{\tau_\epsilon}Y(t)\,dt>h\right)=1.
  \end{align}
\end{lemma} 
Lemma~\ref{ones} follows from the fact that $0$ is an inaccessible
point for the two-dimensional squared Bessel process. The details are
left to the reader.
\begin{proof}[Proof of Lemma~\ref{one}]
  We prove only (\ref{survive}), since
  the proof of (\ref{time0}) is the same (it uses
  (\ref{time}) instead of (\ref{int})). Notice that
\[\lim_{n\to\infty}P_n\left(\sum_{i=0}^{\sigma_0-1}V_i>hn^2\right)\ge
\lim_{\epsilon\to
  0}\lim_{n\to\infty}P_n\left(\sum_{i=0}^{\sigma_{\epsilon n}-1}
  V_i>hn^2\right).\] By the diffusion approximation, for every
$\epsilon\in(0,1)$ \[\lim_{n\to\infty}P_n
\left(\sum_{i=0}^{\sigma_{\epsilon
      n}-1}V_i>hn^2\right)=P^Y_1\left(\int_0^{\tau_\epsilon}Y(t)\,dt
  >h\right),\] and by (\ref{int}),
  \[  \lim_{\epsilon\to 0}P^Y_1\left(\int_0^{\tau_\epsilon}Y(t)\,dt
    >h\right)=1. \qedhere \]
\end{proof}

\section{Continuity of maps $\phi$ and $\psi$}
Denote by $\mathrm{meas}\,A$ the Lebesgue measure of set $A$.
\begin{proposition}
\label{PrPhi}
Let $P$ be a probability measure supported on $C([0, 1])$ such that $P$-a.s.
\begin{equation}\label{zer}
  \mathrm{meas}\{t\in [0,1]: \omega(t)=0\}=0.
\end{equation}
Then the map $\phi$ defined by \eqref{phi} is $P$-a.s.\ continuous.
\end{proposition}

\begin{proof}
  It is sufficient to show continuity at every $\omega\in C([0,1])$
  which satisfies (\ref{zer}). Let $\varpi\in D([0,1])$ and
  $\displaystyle \sup_{t\in [0,1]} |\varpi(t)-\omega(t)|\leq
  \nu$.\footnote{Recall that for $\omega\in C([0,1])$ the Skorokhod convergence
  to $\omega$ implies the uniform convergence (see \cite[the last paragraph on
  p.\,128]{B99}). Thus, it is sufficient to work with the sup norm.} Then
\[ \int_0^1 \mathbbm{1}_{[\nu,\infty)}(\omega(t))\,dt \leq
\phi(\varpi)\leq \int_0^1 \mathbbm{1}_{[-\nu,\infty)}(\omega(t))\,dt\ \ \text{and} \]
\begin{equation}\label{cphi}
  |\phi(\omega)-\phi(\varpi)|\leq \int_0^1 \mathbbm{1}_{[-\nu,\nu]}(\omega(t))\,dt=\meas\{t\in[0,1]:\,-\nu\le \omega(t)\le \nu\}. 
\end{equation}
Since $\{t\in[0,1]:\,-\nu\le \omega(t)\le \nu\}\searrow \{t\in[0,1]:\,\omega(t)=0\}$ and $ \meas\{t\in [0,1]:\, \omega(t)=0\}=0 $, given $\eps>0$ we can choose $\nu>0$ such
that the right-hand side of (\ref{cphi}) is less than $\epsilon$.
\end{proof}

\begin{proposition}
\label{PrPsi}
Let $P$ be a probability measure supported on $C([0, \infty))$ such that $P$-a.s.
\begin{equation}
  \label{zerbel}
  \meas\{t\ge 0:\, \omega(t)=0\}=0\ \text{ and } \ \meas\{t\ge 0:
\omega(t)<0\}=\infty. 
\end{equation}
Then the map $\psi$ defined by \eqref{psi} is $P$-a.s.\ continuous.
\end{proposition}

\begin{proof}
  It is sufficient to show continuity at every
  $\omega\in C([0,\infty))$ which satisfies (\ref{zerbel}). Fix such
  an $\omega$ and let $\epsilon>0$. Recall that
  $T_\omega^-(s):=\inf\{t\ge
  0:\,\meas\{r\in[0,t]:\,\omega(r)<0\}>s\}$.
  Given $R>0$ let $M$ be chosen so that $T_\omega^-(M)=R+1.$ We need
  to find $\nu$ such that if $\varpi\in D([0, \infty))$ satisfies
\begin{equation}
\label{CloseWw}
\sup_{t\in [0, M]} |\varpi(t)-\omega(t)|<\nu
\end{equation}
 then
\begin{equation}
\label{ClosePsi}
 \sup_{t\in [0, R]} |\omega(T_\omega^-(t))-\varpi(T_\varpi^-(t))|<\eps. 
\end{equation} 
We denote $\lim_{s\uparrow t} \varpi(s)$ by $\varpi(t-)$.  Note that due to \eqref{CloseWw} for $t\in (0, M]$ we have
\begin{equation}
\label{CloseWw-}
 |\varpi(t-0)-\omega(t)|<\nu. 
 \end{equation} 
Choose $h$ such that
\begin{equation}
\label{OscW}
 \sup_{t', t''\in [0, M]: |t'-t''|<3 h} |\omega(t')-\omega(t'')|<\eps/8. 
 \end{equation}
Next choose $\nu<\eps/8$ such that 
\begin{equation}
\label{MesNear0}
\meas\{t\in[0,M]:\, |\omega(t)|\le \nu\}<h. 
\end{equation}
Let $\varpi$ satisfy \eqref{CloseWw}. Then for $t\in[0,R]$ we have
\begin{align}
 |\omega(T_\omega^-(t))-\varpi(T_\varpi^-(t))| &\leq |\omega(T_\omega^-(t))-\omega(T_\varpi^-(t))|+
|\omega(T_\varpi^-(t))-\varpi(T_\varpi^-(t))|\nonumber\\
&\leq 
|\omega(T_\omega^-(t))-\omega(T_\varpi^-(t))|+\nu.\label{s1}
\end{align}
For $f\in D([0,\infty))$ let
$A^-_f(t):=\meas\{s\in[0,t]: f(s)<
0\}=\int_0^t\I_{(-\infty,0)}(f(s))\,ds.$
The definition implies that $A^-_f\in C([0,\infty))$ and
$A^-_f(T^-_f(t))\equiv t$.  Note that due to \eqref{CloseWw} we have
\[ A^-_{\omega+\nu}(s) \leq  A^-_\varpi(s) \leq A^-_{\omega-\nu}(s) \]
and due to \eqref{MesNear0} we have
\[ A^-_{\omega-\nu}(s)-h \leq A^-_{\omega}(s)\leq A^-_{\omega+\nu}(s)+h. \]
Therefore,
\[ t-h=A_\omega^-(T^-_\omega(t))-h\leq A_{\omega+\nu}^-(T^-_\omega(t))\leq
A_\varpi^-(T^-_\omega(t))\leq A_{\omega-\nu}^-(T^-_\omega(t))\leq A_\omega^-(T^-_\omega(t))+h=t+h. 
\]
We now consider 4 cases.
\begin{itemize}
\item [(I)] $t-h\leq A^-_\varpi(T^-_\omega(t))\leq t$ (which implies
  that $T^-_\omega(t)\le T^-_\varpi(t)$) and $\omega(u)<0$ for
  $u\in [T^-_\omega(t), T^-_\varpi(t)]$.
\end{itemize}
Then, since $A^-_\varpi(s)-A^-_\varpi(r)\ge s-r-\meas\{u\in[r,s]:\,\varpi(u)\ge 0\}$ for $s\ge r$ and $\varpi(u)\ge 0\ \overset{(\ref{CloseWw})}{\Rightarrow} \ \omega(u)\ge -\nu$ for all $u\in[T^-_\omega(t), T^-_\varpi(t)]$, we have by \eqref{MesNear0} that
\[h\geq A^-_\varpi(T^-_\varpi(t))-A^-_\varpi(T^-_\omega(t))\geq T^-_\varpi(t)-T^-_\omega(t)-h.
\]
Hence, $T^-_\varpi(t)-T^-_\omega(t)\leq 2h$ and so by \eqref{OscW} 
$|\omega(T^-_\varpi(t))-\omega(T^-_\omega(t))|\leq \eps/8. $
\begin{itemize}
\item [(II)] $t-h\leq A^-_\varpi(T^-_\omega(t))\leq t$ and $\omega(\cdot)$ has zeroes on $[T^-_\omega(t), T^-_\varpi(t)].$
\end{itemize}
Let $a$ be the first zero and $b$ be the last zero of $\omega(\cdot)$ on $[T^-_\omega(t), T^-_\varpi(t)]$. Notice that $\omega(T^-_\omega(t))\le 0$. Thus, $\omega(s)\le 0$ for  $s\in[T^-_\omega(t),a]$ and 
the same argument as in case (I) shows that 
\[ |\omega(T^-_\omega(t))|=|\omega(T^-_\omega(t))-\omega(a)|\leq \eps/8. \]
Moreover if $\omega(T^-_\varpi(t))\leq 0$ 
then by the same argument we also have
\[ |\omega(T^-_\varpi(t))|=|\omega(T^-_\varpi(t))-\omega(b)|\leq
\eps/8. \]
On the other hand, if $\omega(T^-_\varpi(t))> 0$ then, since
 $\omega$ is continuous and
$\varpi(T^-_\varpi(t)-))\leq 0$, we get
 \[|\omega(T^-_\varpi(t))|=\omega(T^-_\varpi(t)) \leq \omega(T^-_\varpi(t))- \varpi(T^-_\varpi(t)-)<\nu<\eps/8. \]
In either case we obtain
\[ |\omega(T^-_\varpi(t))-\omega(T^-_\omega(t))|\leq \eps/4. \]
\begin{itemize}
\item [(III)] $t<A^-_\varpi(T^-_\omega(t))\leq t+h$ and $\varpi(u)<0$ for $u\in [T^-_\varpi(t), T^-_\omega(t)].$
\end{itemize}
 Then $h\geq A^-_\varpi(T^-_\omega(t))-A^-_\varpi(T^-_\varpi(t))=
T^-_\omega(t)-T^-_\varpi(t)$,
and so by \eqref{OscW} \[|\omega(T^-_\omega(t))-\omega(T^-_\varpi(t))|\leq \eps/8. \]
\begin{itemize}
\item [(IV)] $t< A^-_\varpi(T^-_\omega(t))\leq t+h$ and $\varpi(\cdot)$ 
takes non-negative values somewhere 
on $[T^-_\varpi(t), T^-_\omega(t)].$
\end{itemize}
Let
\begin{align*}
  a&=\inf\{u\in[T^-_\varpi(t), T^-_\omega(t)]:\, \varpi(u)\ge 0\}\ \text{ and
     }\\ b&=\inf\{u\in[T^-_\varpi(t), T^-_\omega(t)]:\, \varpi(u)\varpi(s)> 0\ \forall s\in[u,T^-_\omega(t))\}.
\end{align*}

Observe that by \eqref{CloseWw} and continuity of $\omega(\cdot)$ it holds that $|\omega(a)|<\nu$ and
$|\omega(b)|<\nu.$ Next, the same argument as in case (III) shows that
$|\omega(a)-\omega(T^-_\varpi(t))|\leq \eps/8.$ Moreover, if
$\varpi(T^-_\omega(t)-)< 0$ then we also have that
$|\omega(T^-_\omega(t))-\omega(b)|\leq \eps/8$, whereas if
$\varpi(T^-_\omega(t)-)\ge  0$ then, since $\omega(T^-_\omega(t))\leq 0$, we
conclude that
\[|\omega(T^-_\omega(t))|\leq |\omega(T^-_\omega(t))-\varpi(T^-_\omega(t)-)|< \nu<\eps/8. \]
Putting everything together we see that in case (IV)
\[ |\omega(T^-_\omega(t))-\omega(T^-_\varpi(t))|\leq
|\omega(T^-_\omega(t))-\omega(b)|+|\omega(b)-\omega(a)|+|\omega(a)-\omega(T^-_\varpi(t))|<\epsilon/2
.\] 

Combining \eqref{s1} with the above estimates for cases (I)--(IV) we obtain that for all $t\in[0,R]$
\[ |\omega(T^-_\omega(t))-\varpi(T^-_\varpi(t))|<5\eps/8. \]
This implies \eqref{ClosePsi} and concludes the proof of the proposition.
\end{proof}

{\bf Acknowledgment:} D.\,Dolgopyat was partially supported by the NSF
grant DMS 1362064.  E.\,Kosygina was partially supported by the Simons
Foundation through a Collaboration Grant for Mathematicians \#209493
and Simons Fellowship in Mathematics, 2014-2015. Parts
  of this work were done during authors' visit to the Fields Institute
  in Spring of 2011. The paper was completed during the second
  author's stay at the Institut Mittag-Leffler in the Fall of
  2014. We thank both Institutes for support and excellent working
  conditions.

\bigskip

{\sc \small
\begin{tabular}{ll}
Department of Mathematics& \hspace*{30mm}Department of Mathematics\\
University of Maryland& \hspace*{30mm}Baruch College, Box B6-230\\
4417 Mathematics Building&\hspace*{30mm}One Bernard Baruch Way\\
College Park, MD 20742, USA &\hspace*{30mm}New York, NY 10010, USA\\
{\verb+dmitry@math.umd.edu+}& \hspace*{30mm}{\verb+elena.kosygina@baruch.cuny.edu+}
\end{tabular}\vspace*{2mm}

}


\begin{thebibliography}{999}

\bibitem{B99} {\sc P.\ Billingsley}\ (1999). Convergence of
  probability measures. Second edition. John Wiley \& Sons, Inc., New
  York, x+277 pp. 

\bibitem{CPY} {\sc R.\ Carmona, F.\ Petit, M.\ Yor}\ (1998). Beta
    variables as times spent in $[0,\infty[$ by certain perturbed
    Brownian motions, \textit{J.\ London Math.\ Soc.\ } {\bf 58}, 239--256.

\bibitem{CD99} {\sc L.\ Chaumont, R.\ A.\ Doney}\ (1999). Pathwise
    uniqueness for perturbed versions of Brownian motion and reflected
    Brownian motion, \textit{Probab.\ Theory Related Fields} {\bf
      113}, no.\ 4, 519--534.

\bibitem{CD00} {\sc L.\ Chaumont, R.\ A.\ Doney}\ (2000). Some
    calculations for doubly perturbed Brownian motion, \textit{
      Stochastic Process. Appl.} {\bf85}, no.\ 1, 61--74.

\bibitem{Da96}
{\sc B.\ Davis}\ (1996). Weak limits of perturbed random walks and the equation
  {$Y_t=B_t+\alpha\sup\{Y_s\colon\ s\leq t\}+\beta\inf\{Y_s\colon\ s\leq t\}$}, \textit{Ann. Probab.} {\bf 24}, no.\ 4, 2007--2023.

\bibitem{DK12} {\sc D.\ Dolgopyat, E.\ Kosygina}\ (2012). Scaling
    limits of recurrent excited random walks on integers,
    \textit{Electron. Commun. Probab.}, {\bf 17}, no.\ 35, 14 pp.

\bibitem{KM} {\sc E.\ Kosygina, T.\ Mountford}\ (2011).
Limit laws of transient excited random walks on
integers. \textit{Ann.\ Inst.\ H.\ Poincar\'e Probab.\ Statist.\ } \textbf{
47}, no.\ 2, 575--600.

\bibitem{KZ08} {\sc E.\ Kosygina, M.\ P.\ W.\ Zerner}\ (2008).  Positively and
  negatively excited random walks on integers,  with branching
  processes, \textit{Electron.\ J.\ Probab.\ } \textbf{13}, no. 64, 1952--1979.

\bibitem{KZ13} {\sc E.\ Kosygina, M.\ P.\ W.\ Zerner}\ (2013).
  Excited random walks: results, methods, open problems.
  \textit{Bull.\ Inst.\ Math.\ Acad.\ Sin.\ (N.S.)}, \textbf{8}. no.\
  1, 105--157.

\bibitem{KZ14} {\sc E.\ Kosygina, M.\ P.\ W.\ Zerner}\ (2014).  Excursions of
  excited random walks on integers \textit{Electron.\ J.\ Probab.\ }
  \textbf{19}, no. 25, 1--25.

\bibitem{PW} {\sc M.\ Perman, W.\ Werner}\ (1997).  Perturbed Brownian
  motions, \textit{Prob.\ Theory Related Fields}, {\bf 108}, no. 3, 357–-383.
\end{thebibliography}
\end{document}